\documentclass{article}
\usepackage[T1]{fontenc}
\usepackage[utf8]{inputenc}
\usepackage[english]{babel}
\usepackage{url}
\usepackage{color}
\usepackage{amsmath,amssymb,amsfonts,amsthm,mathrsfs,stmaryrd}
\usepackage[all]{xy}
\usepackage{hyperref}
\hypersetup{colorlinks,linkcolor=[rgb]{.5,0,0},
  citecolor=[rgb]{.0,.4,.2},urlcolor=[rgb]{.2,.1,.4}}
 \usepackage{unicode}
\newtheorem{thm}{Theorem}[subsection]
\newtheorem{prop}[thm]{Proposition}
\newtheorem{lem}[thm]{Lemma}
\newtheorem{cor}[thm]{Corollary}
\theoremstyle{definition}
\newtheorem{df}[thm]{Definition}
\newtheorem{exs}[thm]{Examples}

\let\ro\mathcal \let\go\mathfrak

\makeatletter
\let\c@equation\c@theorem
\def\theequation{\thesubsection.\arabic{equation}}

\def\letswap#1#2{\let\@tmpa#1\let#1#2\let#2\@tmpa}
 \let\leq\leqslant
 \let\geq\geqslant

\def\makecompact#1{\g@addto@macro#1{%
  \setlength{\itemsep}{\z@}\setlength{\parsep}{\z@}%
  \setlength{\topsep}{\z@}\setlength{\partopsep}{\z@}%
}}

\def\theequation{\@arabic\c@section.\@arabic\c@subsection.\@arabic\c@thm}
\def\endequation{\eqno \hbox{\@eqnnum}$$\@ignoretrue}
\makeatother

\DeclareMathOperator\Sp{Sp}
\DeclareMathOperator\Fil{Fil}
\DeclareMathOperator\Hom{Hom}
\DeclareMathOperator\Ext{Ext}
\DeclareMathOperator\Ker{Ker}
\DeclareMathOperator\Frac{Frac}

\DeclareMathOperator\haut{ht}

\def\O{\ro O}
\def\Qp{{ℚ_p}}
\def\Qph{{ℚ_{p^h}}}
\def\cris{_{\mathrm{cris}}}

\def\Fp{{\mathbb{F}_p}}
\def\Fph{{\mathbb{F}_{p^h}}}
\def\Fpb{{\overline{\mathbb{F}_p}}}
\def\Cp{{ℂ_p}}

\def\trans#1{{}^{\mathrm{t}}#1}

\usepackage{graphicx}
\def\dual#1{#1^{\smash{\scalebox{.7}[1.4]{%
  \rotatebox{90}{\textnormal{\guilsinglleft}}}}}}

\def\defdelim#1#2#3{\def#1##1{{\left#2##1\right#3}}}
\defdelim\pa()
\defdelim\cro[]
\defdelim\acco\{\}
\defdelim\chev<>
\defdelim\abs\vert\vert
\defdelim\norm\Vert\Vert
\defdelim\bcro\llbracket\rrbracket

\def\application#1#2#3#4{\begin{array}{rcl}%
  \displaystyle#1&\longrightarrow&\displaystyle #2\\%
  \displaystyle#3&\longmapsto&\displaystyle #4\\\end{array}}

\makeatletter
\def\@ifdisplay#1#2{\mathchoice{#1}{#2}{#2}{#2}}%
\def\@matharrow#1#2#3{%
  \let\@matharr@short@#1 \let\@matharr@long@#2 \def\@matharr@x@{#3}%
  \let\@matharr@up@\relax \let\@matharr@down@\relax
\@matharr@step}
\def\@matharr@step{\futurelet\@matharr@what \@matharr@}
\def\@matharr@{\let\next\@matharr@do
  \ifx ^\@matharr@what\let\next\@matharr@up \fi
  \ifx _\@matharr@what\let\next\@matharr@down \fi
\next}
\def\@matharr@up^#1{\def\@matharr@up@{#1}\@matharr@step}
\def\@matharr@down_#1{\def\@matharr@down@{#1}\@matharr@step}
\def\@matharr@do{%
  \let\@matharr@do@\@matharr@x
  \ifx\@matharr@up@\relax
    \ifx\@matharr@down@\relax
      \def\@matharr@do@{%
        \@ifdisplay{\@matharr@long@}{\@matharr@short@}}%
    \fi
  \fi
\@matharr@do@}
\def\@matharr@x{%
  \@ifdisplay
  {\@matharr@x@[\mkern 8mu\@matharr@down@\mkern 8mu]%
    {\mkern 8mu\@matharr@up@\mkern 8mu}}%
  {\@matharr@x@[\mkern 12mu\@matharr@down@\mkern 12mu]%
    {\mkern 12mu\@matharr@up@\mkern 12mu}}%
}
\def\longhookrightarrow{\DOTSB\lhook\protect\relbar\protect\joinrel\rightarrow}
\def\hookrightarrowfill@{\arrowfill@{\lhook\mkern 3mu}\relbar\rightarrow}
\renewcommand{\xhookrightarrow}[2][]{\ext@arrow 0399\hookrightarrowfill@{#1}{#2}}

\def\autorightarrow{\@matharrow\rightarrow\longrightarrow\xrightarrow}%
\def\automapsto{\@matharrow\mapsto\longmapsto\xmapsto}%
\def\autohookrightarrow{\@matharrow\hookrightarrow\longhookrightarrow
  \xhookrightarrow}%
\newcommand\xtwoheadrightarrow[2][]{\ext@arrow 0399%
  {\arrowfill@\relbar\relbar\twoheadrightarrow}{#1}{#2}}
\def\autotwoheadrightarrow{\@matharrow\twoheadrightarrow
  \twoheadrightarrow\xtwoheadrightarrow}
\DeclareUnicodeCharacter{21A0}{\autotwoheadrightarrow}

\makeatother

\begin{document}
\makecompact\itemize
\makecompact\enumerate

\title{Analytic $p$-adic Banach spaces and the fundamental lemma of
Colmez and Fontaine}
\date{December 1, 2010}
\author{Jérôme Plût}
\maketitle

\begin{abstract}
This article gives a new proof of the fundamental lemma of the ``weakly
admissible implies admissible'' theorem of Colmez-Fontaine that describes
the semi-stable $p$-adic representations. To this end, we introduce the
category of spectral Banach spaces, which are $p$-adic Banach spaces with
a $\Cp$-algebra of analytic functions, and the subcategory of effective
Banach-Colmez spaces. The fundamental lemma states the surjectivity of
certain analytic maps and describes their kernel. It is proven by an
explicit count of solutions of the equations defining these maps. It is
equivalent to the existence of functions of dimension and height of
effective Banach-Colmez spaces.
\end{abstract}
\section*{Introduction}

The ``weakly admissible implies admissible''
theorem of Colmez and Fontaine~\cite[Théorème~A]{CF2000} states that there
exists an equivalence of categories between the semi-stable $p$-adic
representations and an explicitly described category of filtered
$(φ,N)$-modules.

This theorem rests on a ``fundamental lemma''~\cite[2.1]{CF2000}, of an
analytic nature: let~$U = \acco {x ∈ B^+\cris, φ(x) = px}$, and for~$h ≥
1$, let~$Y = U^h ×_{ℂ_p^h} \Cp$, where~$\Cp → ℂ_p^h$ is any $\Cp$-linear
map, and~$U^h → ℂ_p^h$ is the restriction of the reduction map $θ:
B^+\cris → \Cp$. The fundamental lemma states that any map~$f: Y → \Cp$
of the form~$f(u_1,…,u_h) = u_1v_1 + … + u_h v_h$ (for any~$v_i ∈ B\cris$ such
that~$f$ does map $Y$ to~$\Cp$) is either surjective, or has a
$\Qp$-finite-dimensional image.

This lemma was later improved by Colmez~\cite[6.11]{Colmez2002EBDF}, who proved
that in the case where the map~$f$ is surjective, its kernel has
dimension~$h$ over~$\Qp$. His proof uses sheaves of vector spaces over a
suitable category of $\Cp$-Banach algebras.

\bigbreak

This article gives an independent proof of the strong version of the
fundamental lemma. To this aim, we introduce the new category of spectral
Banach spaces and we see the $\Qp$-vector spaces spaces and linear maps
involved in the theorem as objects and morphisms of this category. The
fundamental lemma is then closely related to the structure of the
subcategory of effective Banach-Colmez spaces. The spectral Banach spaces
are also interesting objects by themselves since, for example, they give
a framework for the Fargues-Fontaine theory~\cite[§8]{FF2011Courbes}.

The spectral Banach spaces are Banach spaces, plus an extra analytic
structure provided by a $\Cp$-Banach algebra of analytic functions. This
category naturally contains all finite-dimensional vector spaces
over~$\Qp$ or~$\Cp$. The objects of the full subcategory of effective
Banach-Colmez spaces are the extensions of finite-dimensional
$\Cp$-vector spaces by finite-dimensional $\Qp$-vector spaces. This is
for example the case of the spaces
\begin{equation}
E_{d,h} = \acco { x ∈ B^+\cris, φ^h(x) = p^d x}
\quad \text{for~$0 ≤ d ≤ h$.}
\end{equation}
These objects are fundamental to the proof of the fundamental lemma, as
it reduces to the two following facts:
\begin{itemize}
\item any effective Banach-Colmez space of dimension one is isomorphic to
the direct sum of~$E_{1,h}$ and a finite-dimensional $\Qp$-vector space;
\item for all~$f_0,…,f_{h-1} ∈ \Cp$, the map~$E_{1,h} → \Cp, x ↦ f_0
θ(x)+…f_{h-1} θ(φ^{h-1}(x))$ (where~$θ: B^+\cris → \Cp$ is the reduction
morphism), is either zero, or surjective with a kernel of dimension~$h$
over~$\Qp$.
\end{itemize}
The fundamental lemma is also interpreted as the existence of natural
functions of dimension and height on the category of effective
Banach-Colmez spaces.

\bigbreak

The first part of this article describes the category of spectral Banach
spaces and explains the analytic structure on some usual objects of
$p$-adic Hodge theory. The second part proves the fundamental lemma. It
first establishes the main properties of the objects~$E_{d,h}$, then the
structure theorem of effective Banach-Colmez spaces and the fundamental
lemma for~$E_{1,h}$ as described above.

\bigbreak

This is part of a work accomplished during a PhD thesis at
Université Paris-Sud 11 (Orsay)
under the supervision of Pr. Jean-Marc Fontaine.

\subsection*{Notations}

Throughout this document, we use the following notations
from~\cite{Fontaine1994Corps}:
$p$~is a prime number,
$ℚ_p$~is the field of $p$-adic integers,
and $ℂ_p$~is the completed algebraic closure of~$ℚ_p$.
For any integer~$h$, $\Qph$~is the unique unramified extension of~$ℚ_p$
of degree~$h$.
We write~$φ$ for the absolute Frobenius automorphism
and $[·]$ for the Teichmüller lift.
We define~$ℤ_p(1)$ as the Tate module of the multiplicative group~$ℂ_p^×$:
its elements are families~$(ε_n)$ such that~$ε_{n+1}^p = ε_n$ and~$ε_0=1$.

For any $p$-adic ring~$B$, $\ro O_B$ is the ring of integers of~$B$.
The multiplication map by~$p$ on~$\ro O_{ℂ_p}/p$ is a ring homomorphism;
we define the ring~$R$ as the projective limit of $\ro O_{ℂ_p}$
for this map.
The projection map~$R → \ro O_{ℂ_p}/p$
extends to a natural map~$θ: W(R) → \ro O_{ℂ_p}$,
where $W(R)$~is the ring of Witt vectors with coefficients in~$R$.
Let~$A\cris$ be the $p$-adic completion of
the divided power hull of~$W(R)$,
relative to the canonical divided powers on the ideal~$\ker θ$;
also define~$B^+\cris = A\cris[\frac 1p]$.
It is a discrete valuation ring, with quotient ring~$ℂ_p$
and maximal ideal generated by an element~$t$ such that~$φ(t) = p· t$.

For any two integers~$d ≥ 0, h ≥ 1$, we define the \emph{slope space}
\begin{equation}
E_{d,h} = \ker \pa{φ^h-p^d: B^+\cris → B^+\cris}.
\end{equation}
For any non-negative rational~$α = \frac{d}{h}$ with $d, h$ coprime,
we also write~$E_{α} = E_{d,h}$; in particular, $E_{d} = E_{d,1}$.
The slope spaces, and in particular~$E_{1,h}$,
play a central role in the proofs given below.
These spaces also appear in~\cite{FF2011Courbes} under the name~$B_E^{φ^{h}=p^d}$.
They are the graded quotients of the structure ring of
the Fargues-Fontaine curve~\cite[9.1,10.1]{FF2011Courbes}.

\section{Spectral Banach spaces}

For the convenience of the reader, we give a self-contained presentation
of spectral Banach spaces. Some results are only stated; complete proofs
can be found in~\cite{These}.
\subsection{Spectral affine varieties}

Let~$p$ be a prime number. Let~$ℚ_p$ be the field of $p$-adic numbers,
$ℂ_p$ be the completion of the algebraic closure of~$ℚ_p$, and~$\O_\Cp$
be the ring of integers of~$\Cp$.

We consider the categories of Banach spaces (and algebras) over~$ℚ_p$ as
topological spaces, that is, up to equivalence of norm. A \emph{lattice}
of a Banach space~$E$ over~$ℚ_p$ is a closed subgroup~$\ro E$ such
that the canonical map~$\ro E ⊗_{ℤ_p} ℚ_p → E$ is an isomorphism.

\begin{df}
Let~$A$ be a topological $\Cp$-algebra. The \emph{spectrum}~$\Sp A$ of~$A$
is the set of all continuous $\Cp$-algebra morphisms from~$A$ to~$\Cp$.
\end{df}

By the Gelfand transform, any element~$a$ of a topological
$\Cp$-algebra~$A$ may be seen as a function on~$S = \Sp A$ by
defining~$a(s) = s(a) ∈ \Cp$ for~$s ∈ \Sp A$. The spectrum is endowed
with the \emph{weak topology}, \emph{i.e.} the coarsest topology for
which all elements of~$a$ are continuous on~$S$.

To each open set~$Ω ⊂ S$ and each~$a ∈ A$, we attach the semi-norm
on~$A$ defined by
\begin{equation}
\norm{a}_{Ω} = \sup \acco {\abs{a(s)}, s ∈ Ω} ∈ [0, + ∞].
\end{equation}
The set~$Ω$ is \emph{bounded} if there exists an open set~$\ro A ⊂ A$ and
a constant~$M < + ∞$ such that~$\norm{a}_{Ω} < M$ for all~$a ∈
\ro A$.

\begin{df}
A topological $\Cp$-algebra~$A$ is \emph{pro-spectral} if $\Sp A$~is
non-empty, $\Sp A$~is the reunion of all bounded open sets~$Ω$, and the topology
on~$A$ is defined by the family of semi-norms~$\norm{·}_{Ω}$. It is
\emph{spectral} if $S = \Sp A$~is bounded and $\norm{·}_{S}$ is a norm defining
the topology on~$A$.
\end{df}

If $A$~is a pro-spectral $\Cp$-algebra, then the natural morphism
from~$A$ to the algebra of continuous functions on~$\Sp A$ is injective.
In particular, $A$~is reduced. Conversely, if $A$~is reduced and
topologically of finite type over~$\Cp$, then it is
spectral~\cite[3.4.9]{FVdP2004}.

For example, the set~$\Cp \acco X$ of formal series with radius of
convergence~$≥ 1$ is a spectral algebra, with spectrum homeomorphic to
the closed ball~$\O_\Cp$~\cite[II.4.4]{FVdP2004}. The set of formal series with
infinite radius of convergence is a pro-spectral algebra, with spectrum
homeomorphic to~$\Cp$.

\begin{df}
The category of \emph{affine spectral varieties} (over~$\Cp$) is the
opposite category to that of pro-spectral algebras (and continuous
$\Cp$-algebra morphisms).
\end{df}

We shall only consider affine spectral varieties, and consequently omit
the word ``affine''.

The forgetful functor from the category of spectral varieties to that of
topological spaces is faithful; therefore, we may see a spectral variety
as a topological space with some extra structure. A morphism of spectral
varieties will also be called \emph{analytic}. For any pro-spectral
$\Cp$-algebra~$A$, we write~$\Sp A$ for the spectral variety attached
to~$A$.

\bigbreak

Let~$S = \Sp A$ be a spectral variety. We say that $S$~is
\begin{itemize}
\item \emph{bounded} if $A$~is spectral;
\item \emph{étale} if $S$~is a locally profinite topological
space, and~$A = \ro C^0 (S, \Cp)$;
\item \emph{connected} if $A$~has no idempotent elements apart
from zero and one;
\item \emph{rigid} if $A$~is topologically of finite type
over~$\Cp$.
\end{itemize}

There is a natural functor from the category of locally profinite
topological spaces to that of spectral varieties. It is fully faithful
and left-adjoint to the forgetful functor: all continuous applications
from an étale spectral variety to any spectral variety are
analytic~\cite[1.4.4]{These}. The only analytic morphisms from a
connected spectral variety to an étale variety are the constant
maps~\cite[1.6.10]{These}.

A morphism of spectral varieties is \emph{surjective} if the underlying
continuous function is surjective. This implies that the corresponding
$\Cp$-algebra morphism is an (always injective) isometry for the spectral
norm~\cite[1.6.1]{These}.

The category of spectral varieties has projective limits for families of
surjective morphisms indexed by~$ℕ$, corresponding to completed unions of
pro-spectral algebras; it also has finite fibre products, corresponding to
completed tensor products of pro-spectral
algebras~\cite[1.6.5,1.6.7]{These}. These constructions are compatible
with the underlying topological spaces.
\subsection{Spectral groups and Banach spaces}

\begin{df}
An (affine, commutative) \emph{spectral group} is a group object in the
category of affine spectral varieties.
\end{df}

Since all groups considered will be affine and commutative, these
adjectives will henceforth be omitted.

The spectral groups inherit from the spectral varieties the existence of
finite fibre products and projective limits of countable surjective
morphisms. Since finite fibre products exist, this category has
kernels. We say that a short sequence of spectral groups is \emph{exact}
if the underlying group sequence is exact.

A \emph{prorigid} spectral group is a projective limit of a countable
family of rigid spectral groups and surjective spectral group morphisms.
Prorigid spectral groups admit the following two properties
(\cite[3.4.4 and 3.4.9]{These}):

\begin{prop}[Global inversion]
Let~$f: G → H$ be a morphism
of prorigid spectral groups. If $f$~is bijective, then it is an
(analytic) isomorphism.
\end{prop}

\begin{prop}[Connected-étale sequence]
Let~$G$ be a prorigid spectral group.
Then $G$~admits a biggest étale quotient~$π_0(G)$ and a biggest connected
subgroup~$G^0$, and there exists a canonical exact sequence
\[ 0 → G^0 → G → π_0(G) → 0. \]
\end{prop}

\medbreak

\begin{df}
An \emph{effective spectral Banach space} is a spectral group~$E$ that is
a Banach space and such that the multiplication map~$\pa{× \frac 1p}: E →
E$ is analytic.
\end{df}

If $E$~is an effective spectral Banach space, then any lattice~$\ro E$
of~$E$ is a spectral group.
An effective spectral Banach space~$E$ is~\emph{prorigid} if there exists
a lattice~$\ro E ⊂ E$ that is prorigid as a spectral group. In this case, in
view of the global inversion theorem, the condition of analyticity of the
multiplication by~$\frac 1p$ is automatically satisfied. Any prorigid
Banach space admits a connected-étale sequence; moreover, this sequence
splits (non-canonically)~\cite[3.5.5]{These}.

\begin{exs}\label{ex:an}
Some examples of effective spectral Banach spaces and spectral groups
are:\penalty 1000
\begin{enumerate}
\item The étale effective spectral Banach spaces are exactly the
finite-dimensional $ℚ_p$-vector spaces~$V$, as well as the set~$c_0(ℚ_p)$
of all convergent sequences~\cite[2.8.2/2]{BGR1984}.
\item Any
finite-dimensional $\Cp$-vector space~$L = ℂ_p^d$ has a canonical structure
as a (connected, prorigid) spectral Banach space, where the analytic
functions are the everywhere convergent formal series in
$d$~variables~\cite[3.3.3]{These}.
\item Since multiplication by~$p$ on~$\O_\Cp$
is analytic and surjective, the projective limit~$R$ is a spectral group.
Its ring structure and the canonical ring morphism~$R → \O_{ℂ_p}/p$
are analytic~\cite[4.2.1]{These}.
\item The ring of Witt vectors~$W(R)$ is spectral, as
well as the $W(R)$-module of Witt bivectors
\begin{equation}
BW(R) = \acco { (x_{n})_{n ∈ ℤ}, \lim \inf_{- ∞} v_R(x_n) > 0 }.
\end{equation}
Moreover, the ring morphism~$θ: BW(R) → \Cp$ and Frobenius $φ: BW(R) → BW(R)$
are analytic~\cite[4.3.4]{These}.
\end{enumerate}
\end{exs}

The ring~$B^+\cris$ does not have a (canonical) analytic structure.
However, there exists a canonical injection~$η: BW(R) ↪ B^+\cris$,
which maps the bivector~$(x_{n})_{n ∈ ℤ}$ to~$∑ p^{-n} x_n^{p^{n}}$;
this map is easily seen to be~$W(R)$-linear and continuous.

\goodbreak
\def\x{\cite[4.4.3]{These}}
\begin{prop}[\x]\label{prop:Edh}
Let~$d ≤ h$ be two integers such that~$0 ≤ d ≤ h$, $h ≥ 1$.
\begin{enumerate}
\item 
$E_{d,h}$~is the set of all bivectors~$(x_n)_{n ∈ ℤ} ∈
BW(R)$ satisfying the periodicity condition~$x_{n-d} = x_n^{p^{h-d}}$
for all~$n ∈ ℤ$.
\item The analytic structure given by this homeomorphism
between~$E_{d,h}$ and~$\go m_R^d$
makes~$E_{d,h}$ an effective spectral Banach subspace of~$BW(R)$.
\item The Frobenius morphism~$φ: E_{d,h} → E_{d,h}$ and the reduction
morphism~$θ: E_{d,h} → \Cp$ are analytic.
\item For any~$c ∈ ℚ_{p^h}$, the multiplication map by~$c$ on~$E_{d,h}$
is analytic;
for any~$d + d' ≤ h$,
the multiplication map~$E_{d,h} × E_{d',h} → E_{d+d',h}$ is analytic.
\end{enumerate}
\end{prop}

%
%
%

\section{Effective Banach-Colmez spaces}
\setcounter{thm}{0}

\begin{df}
The category of~\emph{effective Banach-Colmez spaces} is the full
subcategory of all effective spectral Banach spaces~$E$ that insert in an
analytic short exact sequence
\begin{equation} \label{eq:pres}
0 → V → E → L → 0,
\end{equation}
with $V$~being a finite-dimensional $ℚ_p$-vector space and $L$~being a
finite-dimensional $\Cp$-vector space. Such a short exact sequence is
a~\emph{presentation} of~$E$.

The integers~$h = \dim_{ℚ_p} V$ and~$d =
\dim_{ℚ_p} L$ are called the \emph{height} and \emph{dimension} of this
presentation.
\end{df}

Although effective Banach-Colmez spaces seem to be
a very specific case of spectral Banach spaces,
the logarithmic exponential sequence proves that the group of
$p$-division points of any rigid group has a natural structure as
a Banach-Colmez space~\cite[7.32]{FF2011Courbes}.

The integers~$d$ and~$h$ do \emph{a priori} depend on the choice of the
presentation.
The strong version of the fundamental lemma is equivalent to the fact,
which we prove in this section,
that these integers actually depend only on the effective Banach-Colmez
space~$E$ itself.

We start this section by precising the structure of
some Banach-Colmez spaces via $p$-divisible groups;
we then compute the extension group~$\Ext^1(L, V)$,
which in turns allow to state a structure theorem for
one-dimensional Banach-Colmez spaces.
This allows us to prove the “fundamental lemma” by reducing
to the case where~$E = E_{1/h}$.
\subsection{$p$-divisible groups and effective Banach-Colmez spaces}

Let~$k$ be a perfect subfield of~$\Fpb$
and~$Γ$ be a $p$-divisible group~\cite[II.11]{Demazure1972} over~$k$.
There exists an anti-equivalence of categories between
$p$-divisible groups over~$k$
and \emph{Dieudonné modules}~\cite[III.8]{Demazure1972}
over the ring of~Witt vectors~$W(k)$.
A Dieudonné module is a free module of finite type~$M$ over~$W(k)$, with a
semi-linear map~$φ: M → M$ such that~$φ(M) ⊃ p M$.
We write~$M(Γ)$ for the Dieudonné module associated with
a $p$-divisible group~$Γ$.

\begin{prop}\label{prop:dieud-BW}
Let~$M$ be a Dieudonné module.
\begin{enumerate}
\item For any continuous morphism~$f: M → B^+\cris$ commuting with~$φ$,
the image of~$f$ is contained in~$BW(R)$.
\item The group
\[ E(M) = \Hom_{W(k), φ} (M, B^+\cris) \]
has a canonical structure as a spectral Banach space.
\end{enumerate}
\end{prop}

\begin{proof}
(i) Up to a finite extension~$k'$ of the field~$k$,
$M$~is isomorphic(~\cite[IV.4]{Demazure1972}) to
a direct sum of modules of the form~$M_{d,h} = W(k)[φ]/(φ^h-p^d)$.
Since $E(M ⊗_k k') = E(M)$, its is enough to prove this when~$M = M_{d,h}$.
The condition~$φ(M) ⊃ p M$ then means that~$d ≤ h$,
and thus by~\ref{prop:Edh}(ii), the image of~$f$ lies in~$E_{d,h} ⊂ BW(R)$.

(ii) 
Let~$h$ be the rank of the free module~$M$ over~$W(k)$.
Then $E(M)$~is the kernel of the map~$F: x ↦ x ∘ φ - φ ∘
x$ from~$\Hom_{W(k)} (M, BW(R)) = BW(R)^h$ to itself. The map~$F$ is
analytic by~\ref{prop:Edh}(iii) and therefore its kernel is spectral.
\end{proof}

\begin{prop}\label{prop:EM-GammaR}
Let $Γ$~be a $p$-divisible group over~$k$, of dimension~$d$ and
height~$h$, and let~$M$ be the Dieudonné module of~$Γ$.
Let~$R$ be the complete, perfect ring of characteristic~$p$ defined
in~\ref{ex:an}(iii).
Then the group~$Γ(R)$ has a canonical structure as a
spectral Banach space, and it is canonically isomorphic to
the space~$E(M)$ defined in Prop.~\ref{prop:dieud-BW}.
\end{prop}

\begin{proof}
Let~$d$ be the dimension of~$Γ$.
Since $Γ$~is smooth, by~\cite[II.10,II.11]{Demazure1972},
the choice of a basis of the tangent space to~$Γ$ defines an isomorphism
between the affine algebra of~$Γ$ and
the formal series algebra~$k\bcro{x_1,…,x_d}$.
Such an isomorphism also identifies the set~$Γ(R)$ with~$\go m_R^d$.
Moreover, any other choice of coordinates amounts to a $k$-linear
automorphism of~$\go m_R^d$, which is always analytic.
We thus define the analytic structure on~$Γ(R)$ by transport from~$\go m_R^d$.

By definition of the Dieudonné functor~\cite[III.1.2]{Fontaine1977},
we know that~$Γ(R) = \Hom_{W(k), φ} (M, CW(R))$.
Since~$Γ(R)$~is $p$-divisible, there exists a group isomorphism
\begin{equation}
Γ(R) = \Hom_{W(k), φ} (M, CW(R)) = \Hom_{W(k), φ} (M, BW(R)).
\end{equation}
By~\ref{prop:Edh}(i),
there exists a natural group isomorphism~$η: E(M) → Γ(R)$.

Let~$(e_1,…,e_h)$ be a $W(k)$-basis of~$M$ such that~$(e_1,…,e_d)$ is a
basis of~$M/φ(M)$. Identifying~$E(M)$ with a subspace of~$BW(R)^h$ through
this choice of coordinates, the map~$η$ is then~$(x_1,…,x_h) ↦
(x_{1,0},…,x_{h,0})$, where~$(x_{i,0})$~is the zero component of the Witt
bivector~$x_i$. This map is therefore analytic. Since both spaces are
prorigid, by the global inversion theorem, $η$~is an analytic
isomorphism.
\end{proof}

\begin{prop}\label{prop:Hom-WB}
Let~$B$ be a complete, perfect $k$-algebra.
There exists a natural bijection between the set of continuous algebra
morphisms
\[ \Hom_{W(k),\mathrm{cont}} (W(B), \O_\Cp) \quad ≃ \quad
  \Hom_{k,\mathrm{cont}} (B, R).\]
\end{prop}

\begin{proof}
Let~$f: B → R$ be a continuous $k$-algebra morphism. Since~$B$~is
perfect, there exists a unique sequence of maps~$f_n: B → R$ such
that~$f_0 = f$ and~$f_{n+1}^{p} = f_n$.

Let~$x = (x_n)_{n ≥ 0} ∈ W(B)$. Then the series~$∑ p^{n} θ(f_n(x_n))$
converges to some value~$g(x) ∈ \O_\Cp$, and the map~$g: W(B) → \O_\Cp$
is a continous $W(k)$-algebra morphism. We define a map~$Θ: \Hom (W(B),
\O_\Cp) → \Hom (B, R)$ by defining~$Θ(f) = g$.

Conversely, given a continous morphism~$g:  W(R) → \O_\Cp$ and~$x ∈ B$,
let~$x_n = p^{n} [x]^{p^{-n}}$ be the Witt vector with $n$-th coordinate
equal to~$x$, and define~$f(x) = (g(x_n))_{n ≥ 0}$. We see that the
map~$g ↦ f$ is the inverse of~$Θ$.
\end{proof}

\begin{prop}\label{prop:EM-BCE}
Let~$Γ$ be a connected $p$-divisible group over~$k$ and~$G$ be a smooth
formal group over~$W(k)$ such that~$G ⊗_{W(k)} k = Γ$. Then the following
hold:
\begin{enumerate}
\item The group~$E(M)$~is canonically isomorphic to~$\Hom (ℚ_p, G(\O_\Cp))$.
\item Let~$t_G$ be the tangent space to~$G$ at zero and $T(Γ)$~be the
Tate module~$\Hom (ℚ_p/ℤ_p, Γ)$. Then there exists an
analytic exact sequence
\[ 0 → T(Γ)(\overline k) ⊗_{ℤ_p} ℚ_p → E(M) → t_G (\O_\Cp) ⊗_{ℤ_p} ℚ_p → 0. \]
\end{enumerate}
\end{prop}

In particular, $E(M)$~is an effective Banach-Colmez space.

\begin{proof}
(i) Let~$A$ and~$B$ be the affine algebras of~$G$ and~$Γ$, and~$[p]_A: A
→ A$, $[p]_B: B → B$ be the multiplication by~$p$ in the corresponding
groups. Let~$A'$ and~$B'$ be the projective limits of~$A$ and~$B$ for the
maps~$[p]_A$ and~$[p]_B$.
Since $Γ$~is connected, $[p]_B$ is cofinal to the Frobenius map of~$B$,
and $B'$~is therefore canonically isomorphic to the completion of
the radical closure of~$B$.
The reduction map~$A → B$ then extends to a
map~$A' → B'$, whose reduction modulo~$p$ is an isomorphism~$A' / p A' ≃
B'$. Since~$A'$ is $p$-adically separated and complete and~$B'$ is
perfect, $A'$~is isomorphic to to the ring of Witt vectors~$W(B')$.
Moreover, $A'$ represents, by construction, the group
functor~$U(G) = \Hom(ℚ_p, G)$.

By Proposition~\ref{prop:Hom-WB}, there exists a bijection between
the sets~$\Hom(A',\O_\Cp) = U(G)(\O_\Cp)$ and~$\Hom(B',R) = Γ(R)$.
Moreover, this bijection being functorial in~$Γ$, it is actually a group
isomorphism. Therefore, the three group functors~$U(G)(\O_\Cp)$, $Γ(R)$
and (by~\ref{prop:EM-GammaR})~$E(M)$ are functorially isomorphic. In
particular, the first of those three groups is independent of the
choice of the lift~$G$.

The exact sequence in~\ref{prop:EM-BCE}(ii) is the logarithmic exact
sequence from~\cite[§4]{Tate1967}, tensored with~$\Qp$. Since
$T(Γ)(\overline{k}) ⊗_{ℤ_p} \Qp$~is étale, the left morphism is analytic;
the right morphism is the formal group logarithm of~$G$ and, therefore,
is analytic.
\end{proof}

Let~$M_{d,h} = W(k)[φ]/(φ^h-p^d)$.
Then~$M_{d,h}$~is a Dieudonné module and~$E(M_{d,h}) = E_{d,h}$.
By~\ref{prop:EM-GammaR},
there exists a canonical analytic structure on~$E_{d,h} = E(M_{d,h})$.
Moreover, the choice of a lift of
the $p$-divisible group associated to~$M_{d,h}$
gives a presentation of~$E_{d,h}$ with height~$h$ and dimension~$d$.

\begin{prop}\label{prop:Hom-EM-C}
Let~$M$ be a Dieudonné module.
Then the group of analytical linear applications from~$E(M)$ to~$\Cp$
is canonically isomorphic to~$M ⊗_{W(k)} \Cp$.
\end{prop}

See~\cite[5.2.7]{These} for a complete proof.

\begin{proof}[Sketch of proof]
The map~$F: M ⊗_{W(k)} \Cp → \Hom_{\mathrm{an}} (M, \Cp)$ associates to an
element~$x ⊗ λ$ the analytic function
\[ \application{E(M)}{\Cp}{\pa{f:M\rightarrow B^+\cris}}{λ· θ(f(x))}. \]
To prove that~$F$ is bijective, we may, as in Prop.~\ref{prop:dieud-BW},
assume that there exist integers~$0 ≤ d ≤ h$ such that~$M = M_{d,h}$.
In this case, we identify~$M_{d,h} ⊗ \Cp$ to~$ℂ_p^h$,
and thus~$F(λ_0,…,λ_{h-1}) = λ_0 θ + … + λ_{h-1} θ ∘ φ^{h-1}$.
This map is injective.

Finally, the surjectivity of~$F$ may be proven using an explicit computation
of the analytic structure of~$E(M)$ as a projective limit of rigid Banach
spaces associated to particular submodules of~$M$.
\end{proof}

For any effective Banach-Colmez space~$E$, let~$\dual{E}_\Cp$
be the $\Cp$-vector space of analytic morphisms from~$E$ to~$\Cp$,
and~$E_\Cp$ for the dual $\Cp$-vector space to~$\dual{E}_\Cp$.
The biduality morphism~$ι: E → E_\Cp$~is called the \emph{vector hull} of~$E$.
We see by taking coordinates that any morphism from~$E$ to a $\Cp$-vector space
factorizes uniquely through~$ι$.
In particular, according to proposition~\ref{prop:Hom-EM-C},
for any~$0 ≤ d ≤ h$,
the space~$(E_{d,h})_{ℂ_p}$ is canonically isomorphic to~$ℂ_p^h$,
with the hull map being given by $ι(x) = (θ(φ^{r}(x)))_{r=0,…,h-1}$.
\subsection{The universal extension of $\Cp$ by~$\Qp$}

The Banach-Colmez space~$E_1 = \acco {x ∈ B^+\cris, φ(x) = x}$ has the
canonical presentation~\cite[5.3.7.(ii)]{Fontaine1994Corps}
\begin{equation}\label{eq:pres-E1}
0 → ℚ_p(1) → E_1 →^{θ} \Cp → 0,
\end{equation}
where~$θ$ is the restriction to~$E_1$ of the reduction map~$θ: B^+\cris →
\Cp$.

This extension is universal in the following way.
Let~$V$ be a $h$-dimensional $ℚ_p$-vector space, $L$ be a
$d$-dimensional $\Cp$-vector space, and~$f: L → V ⊗ \Cp$ be a $\Cp$-linear
application. We may then form the fibre
product~$E(f)$ in the following diagram, where the second line
is~(\ref{eq:pres-E1}) tensored with~$V$:
\begin{equation}\label{eq:E(f)}
\xymatrix{
0\ar[r] & V(1)\ar[r] \ar@{=}[d] & E(f) \ar[r] \ar[d] & L \ar[r]
\ar[d]^{f} & 0\\
0\ar[r] & V(1)\ar[r] & V⊗_\Qp E_1\ar[r]^{V ⊗ θ}
  & V⊗_\Qp \Cp\ar[r] & 0
}
\end{equation}
Then~$E(f)$~is an effective spectral Banach space, and~(\ref{eq:E(f)}) is
a presentation of~$E(f)$ as an effective Banach-Colmez space. The Banach
space~$E(f)$ is connected if~$f(L) ∩ V = V$ (or equivalently, if the
transpose map~$\trans{f}: \dual{V}_{ℚ_p} → \dual{L}_{\Cp}$ is injective);
it is étale if~$L = 0$.

\begin{lem}\label{lem:cocycle}
Define a \emph{$p$-extractible} element of~$ℂ_p \acco{X}$
as an element that admits a $p^m$-th root in~$ℂ_p \acco{X}$ for all~$m$.

Let~$g ∈ ℂ_p\acco{X}$ be such that $g(0)=1$
and $g(x)^p/g(px)$ is $p$-extractible.
Then $g$~is the product of an exponential function
and a $p$-extractible element.
\end{lem}
\begin{proof}
%
Write~$g(x)^p/g(px) = u_m(x)^{p^m}$ for all~$m ∈ ℕ$.
Then, for any~$n$, the product $v_n(x) = ∏_{m ≥ 0} u_{n+m+1} (p^m x)$
converges, and $v_n ∈ ℂ_p \acco{X}$.
Moreover, $v_{n+1}^p = v_n$ by construction,
so that $v_0$~is $p$-extractible.
Let~$w = g/v_0$.

Then from the relations~$v_n(x) = u_{n+1}(x) v_{n+1}(px)$
and~$w = v_{n}^{p^{n}}/g$ we deduce~$w(x)^p = w(px)$.
Taking the logarithmic derivative of this relation
yields $(w'/w)(x) = (w'/w) (px)$,
so that~$w'(x) = w(x) w'(0)$ for all~$x$.
Therefore $w$~is an exponential function.
\end{proof}
\begin{prop}\label{prop:extensions}
Any analytic extension of a $d$-dimensional $\Cp$-vector space~$L$ by
a~$h$-dimensional $ℚ_p$-vector space~$V$ is of the form~(\ref{eq:E(f)}).
\end{prop}
\begin{proof}
Let~$(u_i)$ be a basis of~$L$. Then the fibre product~$E_i = E ×_L (\Cp
u_i)$ is an analytic extension of~$\Cp$ by~$V$, and proving the
proposition for all~$E_i$ proves it for~$E$. Therefore, we may assume
that~$L = \Cp$. Moreover, let~$(f_j)$ be a basis of the dual of~$V$.
Then, for each~$j$, the push-out~$E^j = E +_{V} \Qp$ has an analytic
structure, its affine algebra being the quotient of that of~$E$ by the
closed ideal of the functions that are zero on~$\Ker f_j$. Therefore, the
extension~$0 → \Qp → E^j → L → 0$ is analytic, and proving the result for
all~$E^j$ proves it for~$E$.

We thus see that it suffices to prove the case where~$L = \Cp$ and~$V =
ℚ_p$; by multiplication by a suitable power of~$p$, we may even reduce to
the case of extensions of~$\O_\Cp = \Sp \Cp\acco{X}$ by the spectral
group~$ℤ_p$.

Let~$S = \Sp A$ be such an extension.
The diagram~$ℤ_p(1) ↪ S ↠ \O_{ℂ_p}$
corresponds to continuous algebra morphisms
$ℂ_p\acco{X} ↪ A ↠^{ρ} \ro C^0(ℤ_p(1), ℂ_p)$.
Let~$ε = (ε_n)_{n ≥ 0}$ be a topological generator of~$ℤ_p(1)$,
and define~$i_n: ℤ_p -> ℂ_p$ by~$i_n(m) = ε_n^m$.
Then $A$~is topologically generated over~$ℂ_p\acco{X}$
by elements~$f_n$ such that~$ρ(f_n) = i_n$.
Multiplying~$f_n$ by elements of~$ℂ_p\acco{X}$,
we can assume that~$f_{n+1}^p ≡ f_n \pmod{p}$.
Therefore, for all~$n$, the sequence~$(f_{n+m})^{p^m}$
converges to some element~$g_n$ of~$A$.
The family~$(g_n)$ again topologically generates~$A$;
moreover~$g_{n+1}^p = g_n$ by construction,
while~$g_0 ∈ ℂ_p\acco{X}$, $g_0(0) = 1$,
and $g_0$~takes its values in~$\O_{ℂ_p}^{×}$.

For any~$n$, $g_n^p$ and~$g_n(px)$ coincide on~$ℤ_p(1)$,
so that~$ρ(g_n^p/g_n ∘ [p]) = 1$
and~$u_n = g_n(x)^p / g_n(px) ∈ ℂ_p\acco{X}$.
By construction, $g_0(x)^p/g_0(px) = u_0(x)$
and $u_0$~is extractible as defined in Lemma~\ref{lem:cocycle}.
Therefore there exists a sequence~$v_n$ in~$ℂ_p\acco{X}$
and~$a ∈ \ro O_C$ with~$v_p(a) ≥ 1/(p-1)$
such that~$g_0(x) = \exp (a\, x) v_n^{p^n}$.
Replacing $g_n$ by~$v_n g_n$ we may assume that~$g_0(x) = \exp(a\, x)$.%
\footnote{A different way to reach the same result
is to consider instead the cocycle~$\frac{g_0(x+y)}{g_0(x)g_0(y)}$
\cite[II.6.1]{Lazard1975CFG}\cite[5.5.4]{These};
we find it simpler to use its multiplication-by-$p$ analogues instead.}

A point~$s ∈ \Sp A$ is thus determined by the values~$s_n = g_n(s) ∈ \Cp$,
satisfying the condition~$v_p(s_0) ≥ v_p(α)$ and~$s_{n+1}^p = s_n$.
Therefore, $\Sp A$~is homeomorphic to a closed ball of~$\go m_R$.
The map~$\Sp A → B^+\cris, s ↦ \log [s]$ is then an analytic isomorphism
between $\Sp A$ and a lattice of~$E_1$.
\end{proof}
\subsection{Structure of effective Banach-Colmez spaces of dimension one}

The following results give a full description of the category of
Banach-Colmez spaces ``of dimension one'', \emph{i.e.} having a
presentation $0 → V → E → \Cp → 0$. Namely, for these spaces, there
exists an integer~$h$ such that the connected-étale sequence is of the
form
\begin{equation}
0 → E_{1/h} → E → π_0(E) → 0;
\end{equation}
since this sequence splits, the space~$E$ is (non-canonically)
isomorphic to~$E_{1/h} ⊕ π_0(E)$, with~$π_0(E)$ being a
finite-dimensional $\Qp$-vector space.

\begin{prop}\label{prop:structure}
Let~$E$ be an effective Banach-Colmez space having a presentation of
dimension one. Assume that~$E$ is connected and not isomorphic to~$\Cp$.
Then there exists an integer~$h ≥ 1$ such that $E$~is isomorphic
to~$E_{1,h}$.
\end{prop}

The proof uses the following lemma.
\begin{lem}\label{lem:E1h}
Let~$V$ be a $h$-dimensional $ℚ_p$-vector subspace of~$\Cp$. Write~$λ: V →
\Cp$ the canonical injection and~$ι: E_{1/h} → ℂ_p^h$ the vector hull
of~$E_{1/h}$. Then the map~$ι ⊗ λ: E_{1/h} ⊗_{\Qp} V → ℂ_p^h$ is surjective, and
its kernel is a $h$-dimensional vector space over~$ℚ_{p^h}$.
\end{lem}

\begin{proof}
The strategy involves the following steps:
\begin{enumerate}
\item write~$f$ as a left $D_{1/h}$-linear map, where $D_{1/h}$ is
the division algebra over~$\Qp$ with Brauer invariant~$1/h$;
\item compute lattices~$\ro E$ of~$E_{1/h}$ and~$\ro S$ of~$ℂ_p^{h}$
such that~$f(\ro E) ⊂ \ro S$,
and thus reduce the problem modulo~$p$;
\item prove that the reduced map~$\overline{f}: \ro E / π \ro E → \ro S
/ π \ro S$ (where $π$~is a uniformizer of~$D_{1/h}$) is surjective;
\item count the elements of~$\Ker \overline{f}$ and thus proving that
its is one-dimensional over~$\Fph$.
\end{enumerate}
The points~(iii) and~(iv) are proven here only in the special case where
$V$~has a basis consisting of elements~$(λ_i)$ with~$v_p(λ_i) = 0$. The
full proof, similar in spirit but somewhat longer, is detailed
in~\cite[6.2.3]{These}.

\emph{Step (i).}
The division algebra~$D_{1/h}$, having Brauer invariant~$1/h$ over~$ℚ_p$,
is the non-commutative algebra generated over~$\Qph$
by a uniformizer~$π$ satisfying the relations~$π^h=p$
and, for all~$x ∈ \Qph$,~$π ·x = φ(x)· π$.

The division algebra~$D_{1/h}$ acts on $E_{1/h}$ with $π$ acting by the
Frobenius morphism~$φ$, and on~$ℂ_p^h$ by~$π(x_0,…,x_{h-1}) =
(x_1,…,x_{h-1},p x_0)$

Let~$(λ_1, …, λ_h)$ be a basis of~$V$; then the map~$ι ⊗ λ$ may be
written as
\begin{equation}
f: \application{E_{1/h}^h}{ℂ_p^h}
 {(x_1,…,x_h)}{\pa{∑_{i=1}^{h} λ_i θ φ^r(x_i)}_{r=0,…,h-1}}.
\end{equation}
From this, we immediately see that~$f(φ(x)) = π(f(x))$; therefore, the
map~$f$ is $D_{1/h}$-linear.

\emph{Step (ii).}
By selecting an appropriate basis of~$V$ over~$ℚ_p$,
and up to multiplication by powers of~$p$,
we may assume that the~$λ_i$ are sorted by increasing $p$-adic valuation,
and that~$v_p(λ_i) ∈ [0,1[$.
Finally, the reductions of all the~$λ_i$ with same $p$-adic valuation
modulo the appropriate ideal of~$\O_{ℂ_p}$
are linearly independent over~$\Fp$.

For any~$r ∈ ℝ$, define
\begin{equation}
g(ρ) = \min \acco { v_p ( θ(x)),\, x ∈ R, v_R(x) ≥ ρ }
=\min \acco { p^{nh} ρ - n,\, n ∈ ℤ}.
\end{equation}
Then $g: ℝ → ℝ$~is a strictly monotonous function,
piece-wise affine, with slope~$p^{nh}$ on intervals of width~$p^{-nh}$.
This means that it is possible to find elements~$ρ_1, …, ρ_h ∈ ℝ$
such that, for all~$r = 0, …, h-1$, the quantity
$τ_r = g(p^r ρ_i) + v(λ_i)$ does not depend on~$i ∈ \acco{1, …, h}$.
Therefore~$f(\ro E) ⊂ \ro S$,
where $\ro E$ and~$\ro S$ are the lattices of~$E_{1/h}^h$ and~$ℂ_p^h$
defined by
\begin{equation}
\begin{aligned}
\ro E &= \acco{(x_i)_{i=1,…,h} ∈ E_{1/h}^h, v_R(x_{i}) ≥ ρ_i},\\
\ro S &= \acco {(y_r)_{r=0,…,h-1} ∈ ℂ_p^h,
  v_p(y_r) ≥ τ_r}.
\end{aligned}
\end{equation}

%
%
The maximal order~$\ro D_{1/h}$ of~$D_{1/h}$ is
generated by~$ℤ_{p^h}$ and~$π$; it is separated an complete for the
$p$-adic topology and~$\ro D_{1/h} / π \ro D_{1/h} = \Fph$.
Both~$\ro E$ and~$\ro S$
are stable under the action of~$\ro D_{1/h}$, and the map~$f$ reduces to
a $\Fph$-linear map~$\overline{f}: \ro E / φ(\ro E) → \ro S / φ(\ro S)$.
The surjectivity of~$f$ is equivalent to that of~$\overline{f}$, and the
dimension (finite or not) of~$\Ker f$ over~$D_{1/h}$ is equal to the
dimension of~$\Ker \overline{f}$ over~$\Fph$.

\emph{Step (iii).}
For any~$x ∈ \ro E$, only a finite number of terms of the series~$f(x)$
do not belong to~$π(\ro S)$;
therefore, the Lemma reduces to counting the
solutions of a system of polynomials over~$\go m_R$.

We address here only the case where all~$v_i$ are zero.
In this case, the images~$\overline{λ_i}$ of the~$λ_i$
in~$\overline{\Fp}$ are $\Fp$-linearly independent, and we may write
\begin{equation}
\ro E = \acco { x ∈ R^{h}, v_R(x_i) ≥ \frac{1}{p^h -1} },
\qquad \ro S = ⨁_{r=0}^{h-1} p^{\frac{p^{r}}{p^h -1}} \ro O_\Cp.
\end{equation}
Let~$x = (x_1,…,x_h) ∈ \ro E$; then
\begin{equation}
\overline{f}(x) =
  \pa{∑_{i=1}^h λ_i ∑_{n ∈ ℤ} p^{-n} θ\pa{[x^{p^{nh+r}}]}}_{r=0,…,h-1};
\end{equation}
looking at the $r$-th component.
the only non-zero terms of the series modulo~$π \ro S$
are: for~$r = 0$, the terms with~$n ∈ \acco{0,1}$;
for~$r = h-1$, the terms with~$n ∈ \acco{-1,0}$;
for all other~$r$, only the term with~$n = 0$.

For all~$z ∈ \O_\Cp$, let $\widehat{z}$ be an element of~$R$
such that~$\widehat{z}^{(0)} = z$.
For~$b = (b_0,…,b_{h-1}) ∈ \ro S / π \ro S$,
the equation~$\overline{f}(x_1,…,x_h) = b$ is then equivalent to
\begin{equation}\label{eq:fbar-x}
\begin{aligned}
∑ \widehat{λ_i} \pa{x_i + \widehat{p}^{-1} x_i^{p^h}} &= \widehat{b_0},\\
∑ \widehat{λ_i} x_i^{p^r} &= \widehat{b_r} & \text{for $r=1,…,h-2$,}\\
∑ \widehat{λ_i} \pa{\widehat{p} \widehat{x_i}^{p^{-1}} + x_i^{p^{h-1}}}
  &= \widehat{b_{h-1}}.
\end{aligned}
\end{equation}
Since $R$~is a perfect field of characteristic $p$,
the change of variables $x_i = \widehat{p}^{1/(p^h -1)} y_i^p$,
$λ_i = μ_i^{p^h}$
yields the equivalent equations in the variables~$y_i ∈ R$,
for some constants~$c_0, …, c_{h-1} ∈ R$:
\begin{equation}\label{eq:fbar-y}
\begin{aligned}
∑ μ_i^{p^{h-r-1}} y_i &= c_r& \text{for $r=1,…,h-2$,}\\
∑ μ_i^{p^{h-1}} (y_i + y_i^{p^{h}}) &= c_0 ,\\
∑ μ_i^{p^{h}} (ϖ y_i + y_i^{p^{h}}) &= c_{h-1},
\end{aligned}
\end{equation}
where~$ϖ$ is an element of~$R$ such that~$v_R(ϖ) = \frac{p-1}{p}$.
Since the $λ_i \pmod{p}$ are $\Fp$-linearly independent,
we can eliminate $h-2$~variables using the linear equations,
thus reducing to two equations in~$(z_1, z_2)$ of the following form:
\begin{equation}\label{eq:fbar-z}
\begin{aligned}
z_1^{p^{h}} + α_1 z_1 +  α_2 z_2 &= c,\\
z_2^{p^{h}} + ϖ β_1 z_1 + ϖ β_2 z_2 &=  c',\\
\end{aligned}
\end{equation}
where~$α_1β_2 - α_2 β_1$~is a unit in~$R$.
This has exactly $p^{2h}$ solutions in~$\Frac R$,
all of which are integral over~$R$ and therefore belong to~$R$.
Therefore $\overline{f}$~is surjective.

\emph{Step~(iv).}
It remains to compute the dimension of the kernel of~$\overline{f}$.
Equation~\ref{eq:fbar-x} has exactly $p^{2h}$ solutions in~$\ro E$;
this must be divided by the number solutions in~$φ(\ro E)$,
which are the solutions such that~$v_R(x_i) ≥ p/(p^h -1)$.
Let~$ξ ∈ R$ such that~$ξ^{p^{h}-1} = ϖ$;
then the change of variables $z_i = ξ w_i$
linearizes the second equation of~(\ref{eq:fbar-y}).
Therefore this system has exactly~$p^h$ solutions,
so that the kernel of~$\overline{f}$ has dimension~$1$ over~$\Fph$.

Finally, since~$\overline{f}$ is the reduction of~$f$ modulo the maximal
ideal, $f$~is surjective and its kernel is a line over
the division algebra~$D_{1/h}$.
\end{proof}

\begin{proof}[Proof of Proposition~\ref{prop:structure}]
The presentation of~$E$ of dimension one corresponds
by~\ref{prop:extensions} to a $\Cp$-linear map~$f: \Cp → \dual{V}_{\Qp}
⊗_{\Qp} \Cp$ whose transpose~$λ: V → \Cp$ is injective.
Lemma~\ref{lem:E1h} shows that~$ι ⊗ λ: E_{1,h} ⊗ V → ℂ_p^h$ is injective
with kernel of dimension one over the division algebra~$D_{1/h}$. This
means that $\Ker (ι ⊗ λ)$~is generated by an element~$a = (a_1,…,a_h)$,
with~$a_i ∈ E_{1/h}$.

Define~$u: \Qph^h → E_{1/h}^h$ as the linear map~$u(x_0,…,x_{h-1}) = ∑
x_i φ^i(a)$, and~$δ = \det u = \det (φ^r(a_i))_{r,i} ∈ B^+\cris$. Then
$δ$~is also the determinant of the map~$D_{1/h} → \Ker (ι ⊗ λ), c ↦ c ·
a$. Since $a$~is a generator of~$\Ker(ι ⊗ λ)$, we have~$δ ≠ 0$.


The constuction~(\ref{eq:E(f)}) makes~$E$ a sub-space of~$E_1 ⊗ V =
E_1^h$. Let~$t$ be a topological generator of~$ℤ_p(1) ⊂ E_1$.
For any~$x = (x_1, …, x_h) ∈ E_1^h$, define~$g(x) = t^{-1} ∑ a_i
x_i$. The map~$g$ is analytic by construction and verifies~$φ^h(g(x)) =
pg(x)$. It only remains to show that $g(E) ⊂ B^+\cris$, and that~$g: E →
E_{1,h}$ is an isomorphism.

Let~$x = (x_1,…,x_h) ∈ E ⊂ E_1^h$ and~$y = g(x)$.
Then $t\,y = ∑ a_i x_i$, and therefore for any~$r ≥ 0$,
$θ(t\,φ^r(y))= p^{-r} θ(φ^{r}(ty)) = 0$,
which means that~$t φ^{r}(y) ∈ \Fil^1 B^+\cris$.
By~\cite[5.3.7]{Fontaine1994Corps}, this implies that~$y ∈ B^+\cris$.
Since~$φ^r(t^{-1} x_i) = t^{r-1} x_i$,
the vector~$(φ^r(y))$ is the product of the matrix~$(φ^r(a_i))$
with the vector~$(t^{-1} x_i)$.
From~$δ ≠ 0$ we thus deduce that $g$~is injective.
Since $g$~is also a morphism of presentations of effective
Banach-Colmez spaces between~$E$ and~$E_{1/h}$, it is easy to check
that it is surjective, and therefore an analytic isomorphism.
\end{proof}

\begin{cor}
Let~$E$ be a Banach-Colmez space having a presentation~$0 → V → E → \Cp →
0$. Then there exists an integer~$h$ such that the connected-étale
sequence for~$E$ is isomorphic to
\[ 0 → E_{1/h} → E → π_0(E) → 0. \]
\end{cor}

\begin{proof}
This can be directly deduced from applying the
proposition~\ref{prop:structure} to the connected component of~$E$.
\end{proof}

\subsection{Dimension and height}

The original version of the fundamental lemma (\cite[2.1]{CF2000}) examines a
map~$f: Y → \Cp$, where the object~$Y$ is built as a fibre product
of~$\Cp$ and~$E_1^h$ in the same way as in~\ref{prop:extensions} and is
therefore an effective Banach-Colmez space, and the map~$f$ is analytic.
Therefore, it can be deduced from the theorem~\ref{prop:LF}.

Moreover, \ref{prop:LF}~is also equivalent to the strong version of the
fundamental lemma~\cite[6.11]{Colmez2002EBDF}: the category of Banach-Colmez
spaces is equivalent to that of ``Espaces Vectoriels de dimension
finie'', which are sheaves of finite-dimensional vector spaces over
certain Banach algebras over~$\Cp$.

\begin{thm}\label{prop:LF}
Let~$0 → V → E → \Cp → 0$ be a presentation of an effective Banach-Colmez
space with dimension one and height~$h ≥ 0$, and~$f: E → \Cp$ be any
analytic morphism. Then either
\begin{enumerate}
\item $f(E)$ is a finite-dimensional $ℚ_p$-vector space with dimension at
most~$h$, or
\item $f$~is surjective and~$f(E)$ is finite-dimensional over~$ℚ_p$ with
dimension~$h$.
\end{enumerate}
\end{thm}

By Proposition~\ref{prop:structure}, it is enough to prove this
when~$E = E_{1/h}$. Since this spectral space is connected, the
fundamental lemma then takes the shorter form below.

\begin{prop}\label{prop:LF-E1h}
Let~$f: E_{1/h} → \Cp$ be a nonzero analytic morphism. Then~$f$~is
surjective and~$\dim_{ℚ_p} \ker f = h$.
\end{prop}

\begin{proof}
By~\ref{prop:Hom-EM-C}, there exist~$f_0, …,f_{h-1} ∈ \Cp$ such that~$f =
f_0 θ + … + f_{h-1} θ φ^{h-1}$. Replacing~$f$ by~$f ∘ φ^{j}$, we may
assume that~$v_p(f_0) = 0$ and~$v_p(f_i) ≥ -\frac{i}{h}$.

By~\ref{prop:Edh}(i), we may identify~$E_{1,h}$ with
the set of all bivectors~$(x_{n})_{n ∈ ℤ} ∈ BW(R)$
with the periodicity condition $x_{n-1}=x_{n}^{p^{h-1}}$~;
this isomorphism to~$\go m_R$ is analytic
by~\ref{prop:Edh}(ii), so we may identify~$E_{1/h}$ with~$\go m_R$
and write
\begin{equation}
f(x) = ∑_{i ∈ ℤ} f_i θ([x^{p^{i}}])
\end{equation}
where the sequence~$(f_i)$ is extended over~$ℤ$ by~$f_{i+h} = p^{-1} f_i$.
Since this application is $ℚ_p$-linear,
we need only to prove that its image contains a ball of~$\Cp$.

We use here an extension of the theory of Newton polygons
to the ring~$R$.
Define the formal series~$f^+$ and the polynomial~$P$, both with
coefficients in~$\Cp$, as $f^{+}(x) = ∑_{n ≥ 0} f_n x^{p^{n}}$ and~$P(x)
= f_0 x + \dots + f_{h-1} x^{p^{h-1}}$. They verify the functional
equation~$f^+(x) = P(x) + \frac 1p f^+(x^{p^h})$.
We recall that the \emph{Newton polygon} of~$f$~\cite{Robert2000analysis} is
the inferior convex hull of the set~$\acco{(p^{n},v_p(u_n)),\, n ∈ ℕ}$,
and that its slopes are the valuations of the zeroes of~$f$.
Since~$v_p(f_i) ≥ -\frac{i}{h}$ and~$v_p(f_0) = 0$,
it has the vertices~$(p^{nh}, -nh)$ for~$n ∈ ℕ$.

Define~$ρ = \frac 1{h(p-1)}$ and let~$b ∈ \O_\Cp$ such that $v_p(b) ≥ ρ$.
By the theory of Newton polygons, the equation~$f^+(x) = b$ has exactly
one solution~$x_0 ∈ ℂ_p$ such that~$v(x_0) = ρ$.
Starting from~$x_0$, we recursively define a sequence~$(x_i)$ such that~
$f^+(x_i) = p^{-i} b$ and~$\lim \inf v_p(x_{i+1}^{p^h}-x_i) - v_p(x_i) > 0$.
Assume that~$x_i$ is known,
and let~$y ∈ \Cp$ such that~$y^{p^h} = x_{i}$.
Then \( f^+(y) = P(y) + \frac 1p f^+(y^q) = P(y) + p^{-(i+1)}b \).
Let~$x_{i+1} = y + t$;
then $f^+(x_{i+1}) = p^{-(i+1)} b$ if, and only if,
\begin{equation}\label{eq:f(y+t)}
f^+(y+t) = f^+(y) -P(y).
\end{equation}
The Newton polygon for this equation in~$t$ shows that
there exists exactly~$q$ solutions~$t$ such that~$v_p(t) ≥ q^{-1} ρ$.
Choose any of them and define~$x_{i+1} = y + t$.
Then the sequence~$(x_i)$ satisfies
the condition~$v_p(x_{i+1}^q - x_i) ≥ \frac{q-1}{q} ρ$.
Thus, for all~$n$, the sequence~$(x_{n+m}^{p^m})_{m ≥ 0}$ converges to
an element~$x'_n$ of~$\O_\Cp$; and the sequence~$(x'_n)_{n ∈ ℕ}$ is an
element~$x'$ of~$\go m_R$ such that~$f(x') = b$.

It remains to count the dimension of the kernel of~$f$. This is given by
the number of solutions when solving~(\ref{eq:f(y+t)}) for~$t$; since
this number is exactly~$q = p^h$, the kernel is of dimension~$h$
over~$\Qp$.
\end{proof}

As a corollary of the fundamental lemma, we obtain the existence of
natural functions of dimension and height on the category of effective
Banach-Colmez spaces. These functions are additive over short exact
sequences. Moreover, for any $p$-divisible group~$Γ$, the dimension and
height of~$Γ$ as a $p$-divisible group are equal to the dimension and
height of the Banach-Colmez space~$Γ(R)$.

\begin{cor}
Let~$E$ be an effective Banach-Colmez space.
\begin{enumerate}
\item There exists integers~$d = \dim E$ and~$h = \haut E$ such that, for
any presentation~$0 → V → E → L → 0$,
we have~$d = \dim_{\Cp} L$ and $d = \dim_{\Qp} V$.
\item Let~$0 → E' → E → E'' → 0$ be an exact sequence of spectral
Banach-Colmez spaces. Then $\dim(E) = \dim(E') + \dim(E'')$ and $\haut(E)
= \haut(E') + \haut(E'')$.
\end{enumerate}
\end{cor}

These integers are called the \emph{dimension} and \emph{height} of~$E$.

\begin{proof}
We prove (ii) first.
Let~$0 → E' → E → E'' → 0$ be an exact sequence of effective
Banach-Colmez spaces having the presentations
$0 → V → E → L → 0$, $0 → V' → E' → L' → 0$, and $0 → V'' → E'' → L'' →
0$.

Let~$(u_i)$, $(u''_i)$ be bases of $L$ and~$L''$ respectively.
By replacing~$E''$ by~$E'' ×_{L''} (\Cp u_i'')$ and~$E$ by~$E ×_{L} (\Cp
u_i)$, we may assume that~$L = L' = \Cp$. We may then directly
deduce~(ii) from applying the fundamental lemma to the composite map~$E →
E'' → L''$.

Finally, (i) is the special case of~(ii) where~$E' = 0$.
\end{proof}
\bibliographystyle{alpha}
\bibliography{biblio}
\noindent
Written at Université Versailles--Saint-Quentin.\\\noindent
Permanent e-mail address: \texttt{jerome.plut@normalesup.org}

\end{document}